\documentclass{amsart}

\newtheorem{theorem}{Theorem}[section]
\newtheorem{lemma}[theorem]{Lemma}

\newtheorem{proposition}[theorem]{Proposition}
\usepackage{graphicx,amssymb}

\theoremstyle{definition}

\newtheorem{example}[theorem]{Example}

\theoremstyle{remark}
\newtheorem{remark}[theorem]{Remark}

\numberwithin{equation}{section}

\begin{document}

\title[Left-orderability]{Left-orderable fundamental group and Dehn surgery on twist knots}

\author{Ryoto Hakamata}
\address{Graduate School of Education, Hiroshima University,
1-1-1 Kagamiyama, Higashi-hiroshima, Japan 739-8524.}

\author{Masakazu Teragaito}
\address{Department of Mathematics and Mathematics Education, Hiroshima University,
1-1-1 Kagamiyama, Higashi-hiroshima, Japan 739-8524.}
\email{teragai@hiroshima-u.ac.jp}
\thanks{The second author is partially supported by Japan Society for the Promotion of Science,
Grant-in-Aid for Scientific Research (C), 22540088.
}%





\begin{abstract}
For any hyperbolic twist knot in the $3$-sphere,
we show that the resulting manifold by $r$-surgery on the knot has left-orderable
fundamental group if the slope $r$ satisfies the inequality $0\le r \le 4$.
\end{abstract}

\maketitle

\section{Introduction}

A non-trivial group $G$ is said to be \textit{left-orderable\/} if it admits
a strict total ordering which is invariant under left-multiplication.
Thus, if $g<h$ then $fg<fh$ for any $f,g,h\in G$.
Many groups, which arise in topology such as 
orientable surface groups, knot groups, braid groups,
are known to be left-orderable.
In $3$-manifold topology,
it is natural to ask which $3$-manifolds
have left-orderable fundamental groups.
Toward this direction, there is very recent evidence of
connections between Heegaard-Floer homology and left-orderability of
fundamental groups.
More precisely, Boyer, Gordon and Watson \cite{BGW} conjecture
that an irreducible rational homology $3$-sphere is an $L$-space if and
only if its fundamental group is not left-orderable.
An $L$-space is a rational homology $3$-sphere $Y$ whose Heegaard--Floer homology group
$\widehat{HF}(Y)$ has rank equal to $|H_1(Y;\mathbb{Z})|$ (\cite{OS}).
They confirmed the conjecture for several classes of $3$-manifolds including Seifert fibered manifolds, 
Sol-manifolds.
Also, they showed that
if $-4<r<4$ then $r$-surgery on the figure-eight knot
yields a $3$-manifold whose fundamental group is left-orderable.
Later, Clay, Lidman and Watson \cite{CLW}
added the same conclusion for $r=\pm 4$.
Since the figure-eight knot cannot yield $L$-spaces by non-trivial Dehn surgery (\cite{OS}),
these give supporting evidences of the conjecture.

The purpose of this paper is to push forward them to all hyperbolic twist knots.
Any non-trivial twist knot except the trefoil is hyperbolic, and
does not admit non-trivial Dehn surgery yielding $L$-spaces (\cite{OS}).
Hence the following result gives a further supporting evidence
of the conjecture mentioned above.

\begin{theorem}\label{thm:main}
Let $K$ be a hyperbolic twist knot in the $3$-sphere $S^3$ as illustrated in Figure \ref{fig:twistknot}. 
If $0\le r\le 4$, then $r$-surgery on $K$ yields a manifold whose fundamental group is left-orderable.
\end{theorem}

\begin{figure}[ht]
\includegraphics*[scale=0.6]{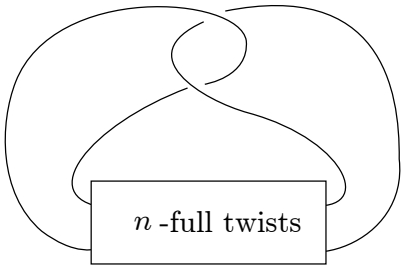}
\caption{The $n$-twist knot}\label{fig:twistknot}
\end{figure}

As seen in Figure \ref{fig:twistknot},
the clasp is left-handed. The range of the slope in the conclusion of Theorem \ref{thm:main} depends on this convention.
If the clasp is right-handed, the range would be $[-4,0]$.

For the right-handed trefoil,
if $r\ge 1$, then $r$-surgery yields an $L$-space (\cite{He}), and its fundamental group is not left-orderable by \cite{BGW}.
Otherwise, $r$-surgery yields a manifold with left-orderable fundamental group
(\cite{CW}).

In \cite{HT0}, we showed the same conclusion as Theorem \ref{thm:main} for the knot $5_2$, which is the $(-2)$-twist knot.
We will greatly generalize the argument in \cite{HT0} to handle all hyperbolic twist knots.
Our argument works even for the figure-eight knot, and
it is much simpler than that in \cite{BGW}, which involves character varieties.
We remark that the fact that the figure-eight knot is amphicheiral
makes possible to widen the range of slope to $-4\le r\le 4$.

\section{Knot group and representations}\label{sec:knotgroup}

Let $K_n$ be the $n$-twist knot with diagram illustrated in Figure \ref{fig:twistknot}.
Our convention is that the twists are right-handed if $n>0$ and left-handed if $n<0$.
Thus $K_1$ is the figure-eight knot and $K_{-1}$ is the right-handed trefoil.
If $n\ne 0, -1$, then $K_n$ is hyperbolic, and
if $|n|>1$, then $K_n$ is not fibered.
Throughout the paper, we assume that $n\ne 0,-1$.
Thus non-trivial Dehn surgery on $K_n$ never yields an $L$-space (\cite{OS}).

It is easy to see from the diagram that $K_n$ bounds a once-punctured Klein bottle
whose boundary slope is $4$.
(For example, consider a checkerboard coloring of the diagram.
Then the bounded surface gives such a once-punctured Klein bottle.)
Thus, $4$-surgery on $K_n$ yields a non-hyperbolic manifold, which is a toroidal manifold.
In \cite{T}, we showed that the resulting toroidal manifold by $4$-surgery on $K_n$
admits a left-ordering on its fundamental group.
Also, $1$, $2$ and $3$-surgeries on $K_n$ are known to yield small Seifert fibered manifolds (\cite{BW}), and the resulting manifolds have left-orderable fundamental groups by \cite{BGW}.
However, we do not need the latter fact.

Let $G=\pi_1(S^3-K_n)$ be the knot group of $K_n$.

\begin{lemma}\label{lem:knotgroup}
The knot group 
$G$ admits a presentation
\[
G=\langle x,y \mid w^nx=yw^n \rangle,
\]
where
$x$ and $y$ are meridians and $w=xy^{-1}x^{-1}y$.

Furthermore, the longitude $\mathcal{L}$ is given as $\mathcal{L}=w_*^nw^n$, where
$w_*=yx^{-1}y^{-1}x$ is obtained from $w$ by reversing the order of letters.
\end{lemma}

This is slightly different from that in \cite[Proposition 1]{HS}, but
they are isomorphic.

\begin{proof}
We use the surgery diagram of $K_n$ as illustrated in Figure \ref{fig:pi1}, where
$1$-surgery and $-1/n$-surgery are performed along the second and third components,
as indicated by numbers, respectively.

\begin{figure}[ht]
\includegraphics*[scale=0.7]{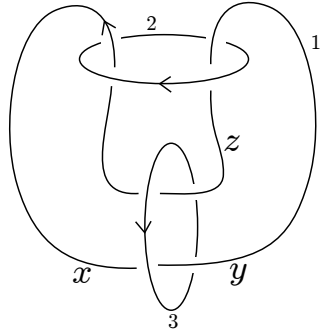}
\caption{A surgery diagram of $K_n$}\label{fig:pi1}
\end{figure}

Let $\mu_i$ and $\lambda_i$ be the meridian and longitude of the $i$-th component.
Then $y=\mu_3^{-1}x\mu_3$, $z=\mu_2^{-1}y\mu_2$, $\lambda_2=x^{-1}y$ and $\lambda_3=yz^{-1}$.
By $1$-surgery on the second component, 
the relation $\lambda_2\mu_2=1$ arises.
Similarly, $-1/n$-surgery on the third component implies $\lambda_3^n\mu_3^{-1}=1$.

Hence, $\mu_3^{-1}=\lambda_3^{-n}=(zy^{-1})^n=(x^{-1}yxy^{-1})^n$.
Let $w=xy^{-1}x^{-1}y$.
Then $x^{-1}yxy^{-1}=x^{-1}ywy^{-1}x$, so $\mu_3^{-1}=x^{-1}yw^ny^{-1}x$.
By substituting this to the remaining relation $y=\mu_3^{-1}x\mu_3$,
we obtain $w^nx=yw^n$.

Finally, the longitude $\mathcal{L}$ is given as $\mu_3\mu_2\mu_3^{-1}\mu_2^{-1}=x^{-1}yw^{-n}y^{-1}xw^n$.
Since $x^{-1}yw^{-n}y^{-1}x=(yx^{-1}y^{-1}x)^n$,
we have $\mathcal{L}=w_*^nw^n$, where $w_*=yx^{-1}y^{-1}x$.
\end{proof}

Let $s>0$ and $t>1$ be real numbers.
Let $\rho_s: G\to SL_2(\mathbb{R})$ be the representation defined by the correspondence
\begin{equation}\label{eq:rho}
\rho_s(x)=
\begin{pmatrix}
\sqrt{t} & 0 \\
0 & \frac{1}{\sqrt{t}}
\end{pmatrix},\quad
\rho_s(y)=
\begin{pmatrix}
\frac{t-s-1}{\sqrt{t}-\frac{1}{\sqrt{t}}} & \frac{s}{(\sqrt{t}-\frac{1}{\sqrt{t}})^2}-1 \\
-s & \frac{s+1-\frac{1}{t}}{\sqrt{t}-\frac{1}{\sqrt{t}}}
\end{pmatrix}.
\end{equation}

For $P=\begin{pmatrix}
t-1 & 1\\
0 & \sqrt{t}-\frac{1}{\sqrt{t}}
\end{pmatrix}$,
\[
P^{-1} \rho_s(x) P=\begin{pmatrix}
\sqrt{t} & \frac{1}{\sqrt{t}} \\
0 & \frac{1}{\sqrt{t}}
\end{pmatrix},\quad 
P^{-1} \rho_s(y) P=\begin{pmatrix}
\sqrt{t} & 0\\
-s\sqrt{t} & \frac{1}{\sqrt{t}}
\end{pmatrix}.
\]
Hence, (\ref{eq:rho}) gives a (non-abelian) representation if
$s$ and $t$ satisfy Riley's equation $z_{1,1}+(1-t)z_{1,2}=0$,
where $z_{i,j}$ is the $(i,j)$-entry of the matrix $P^{-1}\rho_s(w^n)P$ (\cite{R}).
See also \cite{DHY}.
Then $\phi_n(s,t)=z_{1,1}+(1-t)z_{1,2}$ is called the Riley polynomial of $K_n$.

Since $s$ and $t$ are limited to be positive real numbers in our setting, it is not obvious that
there exist solutions for Riley's equation $\phi_n(s,t)=0$.
However, this will be verified in Proposition \ref{prop:root}.
We temporarily assume that $s$ and $t$ are chosen
so that $\phi_n(s,t)=0$.

From (\ref{eq:rho}), we have 
\begin{equation}\label{eq:www}
W=\rho_s(w)=\begin{pmatrix}
1+s-st+\frac{s^2 t}{t-1} &  \frac{t-1+st}{\sqrt{t}}\frac{(\sqrt{t}-\frac{1}{\sqrt{t}})^2-s}{(\sqrt{t}-\frac{1}{\sqrt{t}})^2}\\
\frac{s(1+s-t)}{\sqrt{t}} & 1+s-\frac{s^2}{t-1}-\frac{s}{t} 
\end{pmatrix}.
\end{equation}
Let 
\begin{equation}\label{eq:eigenvalue}
\lambda_{\pm}=\frac{1}{2}\left\{s^2-\left(t+\frac{1}{t}-2\right)s+2\pm \sqrt{\left(s^2-\left(t+\frac{1}{t}-2\right)s+2\right)^2-4}\right\}\in \mathbb{C}.
\end{equation}
These are eigenvalues of $W$, and so $\lambda_++\lambda_-=\mathrm{tr}(W)=s^2-(t+1/t-2)s+2$
and $\lambda_+\lambda_-=1$.
In Proposition \ref{prop:root}, we will see that
$s+2<t+1/t<s+2+4/s$.
This implies 
\[
-2<s^2-\left(t+\frac{1}{t}-2\right)s+2<2,
\]
and so $\lambda_\pm=e^{\pm i\theta}$ for some $\theta\in (0,\pi)$.
In particular, we remark that $\lambda_+\ne \lambda_-$.

\begin{proposition}
The Riley polynomial $\phi_n(s,t)$ of $K_n$ is given by
\begin{equation}\label{eq:riley}
\frac{\lambda_+^{n+1}-\lambda_-^{n+1}}{\lambda_+-\lambda_-}
-\left(t+\frac{1}{t}-1-s\right)\frac{\lambda_+^{n}-\lambda_-^{n}}{\lambda_+-\lambda_-}.
\end{equation}
\end{proposition}

\begin{proof}
The Riley polynomial is explicitly calculated in \cite[Proposition 3.1]{M}.
Our knot $K_n$ corresponds to the mirror image of $J(2,-n)$ in \cite{M}.
This gives the conclusion.
\end{proof}

By using Lemma \ref{lem:wij},
it is not hard to check directly that 
$\rho_s(w^nx)=\rho_s(yw^n)$ holds if and only if 
$s$ and $t$ make the polynomial (\ref{eq:riley}) equal to zero.

Set $T=t+1/t$ and $\tau_m=(\lambda_+^{m}-\lambda_-^{m})/(\lambda_+-\lambda_-)$ for convenience.
Then the Riley polynomial of $K_n$ is expressed simply as
$\phi_n(s,T)=\tau_{n+1}-(T-1-s)\tau_n$.
Since $\tau_m$ is symmetric in $\lambda_+$ and $\lambda_-$,
it can be expressed as a polynomial of $\lambda_++\lambda_-$, which is
$s^2-(T-2)s+2$.
Also, it is easy to see that a recursive relation
\begin{equation}\label{eq:rec}
\tau_{m+1}-(\lambda_++\lambda_-)\tau_m+\tau_{m-1}=0
\end{equation}
and $\tau_{-m}=-\tau_m$ hold for any integer $m$.

\begin{example}\label{ex:riley}
Clearly, $\tau_0=0$ and $\tau_1=1$.
Thus we have $\tau_2=s^2-(T-2)s+2$ and $\tau_3=(s^2-(T-2)s+2)^2-1$.
From these, the figure-eight knot has the Riley polynomial
\[
\phi_1(s,T)=\tau_2-(T-1-s)\tau_1=-(s+1)T+s^2+3s+3.
\]
Similarly, the $2$-twist knot, $6_1$ in the knot table, has the Riley polynomial
\begin{eqnarray*}
\phi_2(s,T)&=& \tau_3-(T-1-s)\tau_2\\
&=& (s^2+s)T^2-(2 s^3+6 s^2+7 s+2)T+s^4+5 s^3+11 s^2+12 s+5.
\end{eqnarray*}
\end{example}

From the recursive relation (\ref{eq:rec}),
we see that the Riley polynomial $\phi_n(s,T)$ has degree $|n|$ in $T$.
Thus we cannot solve the equation $\phi_n(s,T)=0$ for $T$, in general.




\section{Riley polynomials}

In this section, we show that
Riley's equation $\phi_n(s,T)=0$ has a pair of solutions $(s,T)$
such as $s+2<T<s+2+4/s$ for any $s>0$.
In fact, we can choose $T$ satisfying
$s+2+c/s<T<s+2+4/s$ where $c$ is a constant depending only $n$, unless $n=1$.

Let $m$ be a positive integer.
For $z=e^{i\theta}$ $(0\le\theta\le\pi)$,
set $\mathcal{T}_m(z)=z^{m-1}+z^{m-3}+\dots+z^{3-m}+z^{1-m}$.
If $z\ne \pm 1$, then $\mathcal{T}_m(z)=(z^m-z^{-m})/(z-z^{-1})$.
Define $\mathcal{T}_0=0$ and $\mathcal{T}_{-m}(z)=-\mathcal{T}_m(z)$.
Since $\mathcal{T}_m(z)$ is symmetric for $z$ and $z^{-1}$, it can be expanded as a polynomial
of $z+z^{-1}$.
Furthermore, the recursive relation
\[
\mathcal{T}_{m+1}(z)-(z+z^{-1})\mathcal{T}_m(z)+\mathcal{T}_{m-1}(z)=0
\]
holds.
Also, $\mathcal{T}_m(1)=m$ 
and $\mathcal{T}_m(-1)=(-1)^{m-1}m$
for any integer $m$.

\begin{lemma}\label{lem:tau}
\begin{itemize}
\item[(1)] Let $m\ge 1$. Then,
$\mathcal{T}_{m}(e^{\frac{\pi}{2m+1}i})=\mathcal{T}_{m+1}(e^{\frac{\pi}{2m+1}i})$, and
this value is positive.
\item[(2)] Let $m\ge 2$. Then,
$\mathcal{T}_{m}(e^{\frac{3\pi}{2m+1}i})=\mathcal{T}_{m+1}(e^{\frac{3\pi}{2m+1}i})$, and
this value is negative.
\end{itemize}
\end{lemma}

\begin{proof}
(1) Let $z=e^{\frac{\pi}{2m+1}i}$.
Then the fact that $z^{2m+1}=-1$ immediately implies
$\mathcal{T}_m(z)=\mathcal{T}_{m+1}(z)$.
A direct calculation shows
\[
\mathcal{T}_m(z)=\frac{\sin \frac{m\pi}{2m+1}}{\sin\frac{\pi}{2m+1}}>0.
\]

(2) Similarly, set $z=e^{\frac{3\pi}{2m+1}i}$.
Then $z^{2m+1}=-1$ holds again.
Hence we have $\mathcal{T}_m(z)=\mathcal{T}_{m+1}(z)$,
and
\[
\mathcal{T}_m(z)=\frac{\sin \frac{3m\pi}{2m+1}}{\sin\frac{3\pi}{2m+1}}<0.
\]
\end{proof}

Now, fix an $s>0$.
We introduce a function $\Phi_n: [s+2,s+2+4/s]\to\mathbb{R}$
by
\begin{equation}\label{eq:riley-P}
\Phi_n(T)=\mathcal{T}_{n+1}(z)-(T-1-s)\mathcal{T}_n(z),
\end{equation}
where
\[
z=\frac{1}{2}\left\{s^2-(T-2)s+2+i\sqrt{4-(s^2-(T-2)s+2)^2}\right\}.
\]
Since $s+2\le T\le s+2+4/s$, we have
$-2\le s^2-(T-2)s+2\le 2$.
Thus $z=e^{i\theta}$ for $\theta\in [0,\pi]$.
We will seek a solution $T$ for $\Phi_n(T)=0$ satisfying
$s+2<T<s+2+4/s$, because
it gives a pair of solutions $(s,T)$ for Riley's equation $\phi_n(s,T)=0$.

\begin{proposition}\label{prop:root}
Riley's equation $\phi_n(s,T)=0$ has a real solution $T$ satisfying
$s+2<T< s+2+4/s$ for any $s>0$.
Moreover, if $n\ne 1$, then $T$ can be chosen so that
$s+2+c/s<T<s+2+4/s$, where $c$ is a constant depending only on $n$.
In particular, $\phi_n(s,t)=0$ has a solution $t>1$ for any $s>0$. 
\end{proposition}

\begin{proof}
Suppose $n>1$.
By Lemma \ref{lem:tau},
\[
\mathcal{T}_{n+1}(e^{\frac{\pi}{2n+1}i})=\mathcal{T}_n(e^{\frac{\pi}{2n+1}i}),\quad
\mathcal{T}_{n+1}(e^{\frac{3\pi}{2n+1}i})=\mathcal{T}_n(e^{\frac{3\pi}{2n+1}i}).
\]
Let $c=2-2\cos\frac{\pi}{2n+1}$ and $c'=2-2\cos\frac{3\pi}{2n+1}$.
Then $c,c'\in (0,4)$ and $c<c'$.

Also,
\begin{eqnarray*}
\Phi_n(s+2+\frac{c}{s})&=&\mathcal{T}_{n+1}(e^{\frac{\pi}{2n+1}i})-\left(1+\frac{c}{s}\right)\mathcal{T}_n(e^{\frac{\pi}{2n+1}i})\\
&=& -\frac{c}{s}\cdot\mathcal{T}_n(e^{\frac{\pi}{2n+1}i}),\\
\Phi_n(s+2+\frac{c'}{s})&=&\mathcal{T}_{n+1}(e^{\frac{3\pi}{2n+1}i})-\left(1+\frac{c'}{s}\right)\mathcal{T}_n(e^{\frac{3\pi}{2n+1}i})\\
&=& -\frac{c'}{s}\cdot\mathcal{T}_n(e^{\frac{3\pi}{2n+1}i}).
\end{eqnarray*}
By Lemma \ref{lem:tau}, these values have distinct signs.
We remark that $\Phi_n(T)$ is a polynomial function of $T$, so it is continuous.
Thus
if $n>1$, we have a solution $T$ for $\Phi_n(T)=0$, satisfying $s+2+c/s<T<s+2+c'/s$,
from the Intermediate-Value Theorem.
Since $T>2$, $t+1/t=T$ has a real solution for $t$.
If we choose $t=(T+\sqrt{T^2-4})/2$, then $t>1$.

When $n=1$,
the Riley polynomial is
$\phi_1(s,T)=-(s+1)T+s^2+3s+3$ as shown in Example \ref{ex:riley}.
Hence the equation $\phi_1(s,T)=0$ has the unique solution
$T=(s^2+3s+3)/(s+1)=s+2+1/(s+1)$ for a given $s$.
This satisfies $s+2<T<s+2+1/s$.

Suppose $n<0$.  (Recall that we assume $n\ne -1$.)   Set $l=|n|\ge 2$.
If $l>2$, then 
set $d=2-2\cos\frac{\pi}{2l-1}$
and $d'=2-2\cos\frac{3\pi}{2l-1}$.
Then $d,d'\in (0,4)$ and $d<d'$.
As before,
\[
\mathcal{T}_{l-1}(e^{\frac{\pi}{2l-1}i})=\mathcal{T}_l(e^{\frac{\pi}{2l-1}i}),\quad
\mathcal{T}_{l-1}(e^{\frac{3\pi}{2l-1}i})=\mathcal{T}_l(e^{\frac{3\pi}{2l-1}i}).
\]
by Lemma \ref{lem:tau}.
Thus
\begin{eqnarray*}
\Phi_n(s+2+\frac{d}{s})
&=&-\left(\mathcal{T}_{l-1}(e^{\frac{\pi}{2l-1}i})-\left(1+\frac{d}{s}\right)\mathcal{T}_{l}(e^{\frac{\pi}{2l-1}i})\right)\\
&=& \frac{d}{s}\cdot\mathcal{T}_l(e^{\frac{\pi}{2l-1}i}),\\
\Phi_n(s+2+\frac{d'}{s})
&=& \frac{d'}{s}\cdot\mathcal{T}_l(e^{\frac{3\pi}{2l-1}i}).
\end{eqnarray*}
Since these values have distinct signs,
we have a solution $T$ with $s+2+d/s<T<s+2+d'/s$, if $l>2$, as before.

When $l=2$, 
we have
\[
\Phi_{-2}(s+2+\frac{1}{s})=\frac{1}{s}>0,\quad  \Phi_{-2}(s+2+\frac{2}{s})=-1<0.
\]
Hence there exists a solution $T$ with $s+2+1/s<T<s+2+2/s$.
\end{proof}

\section{Longitudes}

Recall that $\rho_s:G\to SL_2(\mathbb{R})$ is the representation
defined by (\ref{eq:rho}).
Two real parameters $s$ and $t$ are chosen so that $\phi_n(s,t)=0$.
In this section, we examine the image of the longitude $\mathcal{L}$ of $G$
under $\rho_s$.
Throughout the section,
let 
\[
\rho_s(w)=\begin{pmatrix}
w_{1,1} & w_{1,2} \\
w_{2,1} & w_{2,2}
\end{pmatrix},\quad 
\rho_s(w^n)=\begin{pmatrix}
u_{1,1} & u_{1,2} \\
u_{2,1} & u_{2,2}
\end{pmatrix}
\]
and 
$\sigma=\frac{s(\sqrt{t}-\frac{1}{\sqrt{t}})^2}{(\sqrt{t}-\frac{1}{\sqrt{t}})^2-s}$.

\begin{lemma}\label{lem:ww*}
For $w_*^n$, we have
$\rho_s(w_*^n)=\begin{pmatrix}
u_{1,1} & \frac{u_{2,1}}{\sigma}\\
u_{1,2}\sigma & u_{2,2}
\end{pmatrix}$.
\end{lemma}

\begin{proof}
By a direct calculation,
\[
\rho_s(xy^{-1})=\begin{pmatrix}
\frac{t-1+st}{t-1} & \frac{s\sqrt{t}}{\sigma} \\
\frac{s}{\sqrt{t}} & \frac{t-1-s}{t-1}
\end{pmatrix}, \quad
\rho_s(y^{-1}x)=\begin{pmatrix}
\frac{t-1+st}{t-1} & \frac{s}{\sqrt{t}\sigma} \\
s\sqrt{t} & \frac{t-1-s}{t-1}
\end{pmatrix},
\]
\[
\rho_s(x^{-1}y)=\begin{pmatrix}
\frac{t-1-s}{t-1} & -\frac{s}{\sqrt{t}\sigma} \\
-s\sqrt{t} & \frac{t-1+st}{t-1}
\end{pmatrix}, \quad
\rho_s(yx^{-1})=\begin{pmatrix}
\frac{t-1-s}{t-1} & -\frac{s\sqrt{t}}{\sigma} \\
-\frac{s}{\sqrt{t}} & \frac{t-1+st}{t-1}
\end{pmatrix}.
\]
Thus we see that the $(1,2)$-entry of $\rho_s(y^{-1}x)$ is
the $(2,1)$-entry of $\rho_s(xy^{-1})$ divided by $\sigma$,
the $(2,1)$-entry of $\rho_s(y^{-1}x)$ is
the $(1,2)$-entry of $\rho_s(xy^{-1})$ multiplied by $\sigma$,
and the others of $\rho_s(y^{-1}x)$ coincide with those of $\rho_s(xy^{-1})$.
The same relation between entries holds for $\rho_s(x^{-1}y)$ and $\rho_s(yx^{-1})$.

In general, such a relation is preserved under the matrix multiplication;
\[
\begin{pmatrix}
a & b \\
c & d
\end{pmatrix}
\begin{pmatrix}
p & q \\
r & s
\end{pmatrix}
=
\begin{pmatrix}
ap+br & aq+bs \\
cp+dr & cq+ds
\end{pmatrix},
\]
\[
\begin{pmatrix}
p & \frac{r}{\sigma} \\
q\sigma & s
\end{pmatrix}
\begin{pmatrix}
a & \frac{c}{\sigma} \\
b\sigma & d
\end{pmatrix}
=
\begin{pmatrix}
ap+br & \frac{cp+dr}{\sigma} \\
(aq+bs)\sigma & cq+ds
\end{pmatrix}.
\]
Thus we can confirm that the same relation holds for $\rho_s(w^n)$ and $\rho_s(w_*^n)$.
\end{proof}

\begin{proposition}\label{prop:longi}
For the longitude $\mathcal{L}$ of $G$,
the matrix $\rho_s(\mathcal{L})$ is diagonal, and
the $(1,1)$-entry of $\rho_s(\mathcal{L})$ is a positive real number.
\end{proposition}

\begin{proof}
The first assertion follows from the facts that
for a meridian $x$, 
$\rho_s(x)$ is diagonal but $\rho_s(x)\ne \pm I$ and that $x$ and $\mathcal{L}$ commute.

Since $\mathcal{L}=w_*^n w^n$ by Lemma \ref{lem:knotgroup}, Lemma \ref{lem:ww*} implies that
\begin{eqnarray*}
\rho_s(\mathcal{L})=\rho_s(w_*^n)\rho_s(w^n)&=&
\begin{pmatrix}
u_{1,1} & \frac{u_{2,1}}{\sigma} \\
u_{1,2}\sigma & u_{2,2}
\end{pmatrix}
\begin{pmatrix}
u_{1,1} & u_{1,2} \\
u_{2,1} & u_{2,2}
\end{pmatrix}
\\
&=&
\begin{pmatrix}
u_{1,1}^2+\frac{u_{2,1}^2}{\sigma} & u_{1,1}u_{1,2}+\frac{u_{2,1}u_{2,2}}{\sigma}\\ 
u_{1,1}u_{1,2}\sigma+u_{2,1}u_{2,2} & u_{1,2}^2\sigma+u_{2,2}^2
\end{pmatrix}.
\end{eqnarray*}

Since $\det \rho_s(w^n)=1$, at least one of $u_{1,1}$ and $u_{2,1}$ is non-zero.
Hence the $(1,1)$-entry is $u_{1,1}^2+u_{2,1}^2/\sigma$, which
is positive, because
$s>0$ and $(\sqrt{t}-1/\sqrt{t})^2-s=T-s-2>0$ from Proposition \ref{prop:root}.
\end{proof}

\begin{remark}\label{rem:longi}
Since $\rho_s(\mathcal{L})$ is diagonal,
we also obtain an equation $u_{1,1}u_{1,2}\sigma+u_{2,1}u_{2,2}=0$.
This will be used in the proof of Lemma \ref{lem:bs}.
\end{remark}


To diagonalize $W=\rho_s(w)$,
let $Q=\begin{pmatrix}
w_{1,2} & w_{1,2} \\
\lambda_+-w_{1,1} & \lambda_--w_{1,1}
\end{pmatrix}.$
From (\ref{eq:www}),
$w_{1,2}=(t-1+st)s/(\sigma\sqrt{t})$.
Since $s>0,t>1, \sigma>0$, we have $w_{1,2}\ne 0$.
Also, $\det Q=-w_{1,2}(\lambda_+-\lambda_-)$.
Then a direct calculation shows
$Q^{-1}WQ=\begin{pmatrix}
\lambda_+ & 0 \\
0 & \lambda_-
\end{pmatrix}$.

\begin{lemma}\label{lem:wij}
The entries of $W^n$ are given as follows.
\begin{eqnarray*}
u_{1,1}&=& w_{1,1}\tau_n-\tau_{n-1},\quad
u_{1,2}= w_{1,2}\tau_n,\\
u_{2,1}&=& w_{2,1}\tau_n,\quad
u_{2,2}= \tau_{n+1}-w_{1,1}\tau_n.
\end{eqnarray*}
\end{lemma}

\begin{proof}
This easily follows from $W^n=Q\begin{pmatrix}
\lambda_+^n & 0 \\
0 & \lambda_-^n 
\end{pmatrix}Q^{-1}$.

For example,
\begin{eqnarray*}
u_{2,1}&=&\frac{1}{\det Q}\Bigl(\lambda_+^n(\lambda_+-w_{1,1})(\lambda_--w_{1,1})+\lambda_-^n(\lambda_--w_{1,1})(w_{1,1}-\lambda_+)\Bigr)\\
&=&-\frac{\tau_n}{w_{1,2}}(\lambda_+-w_{1,1})(\lambda_--w_{1,1})\\
&=&-\frac{\tau_n}{w_{1,2}}(1-\mathrm{tr}(W)w_{1,1}+w_{1,1}^2).
\end{eqnarray*}
Since $\mathrm{tr}(W)=w_{1,1}+w_{2,2}$, we have
$1-\mathrm{tr}(W)w_{1,1}+w_{1,1}^2=1-w_{1,1}w_{2,2}=-w_{1,2}w_{2,1}$.
Thus $u_{2,1}=w_{2,1}\tau_n$.

We omit the others.
\end{proof}

\begin{lemma}\label{lem:bs}
Let $B_s$ be the $(1,1)$-entry of the matrix $\rho_s(\mathcal{L})$.
Then $B_s=-w_{2,1}/(w_{1,2}\sigma)$.
\end{lemma}

\begin{proof}
As noted in Remark \ref{rem:longi}, $u_{1,1}u_{1,2}\sigma+u_{2,1}u_{2,2}=0$.
Since $\det W^n=u_{1,1}u_{2,2}-u_{1,2}u_{2,1}=1$,
we have
\begin{eqnarray*}
u_{1,2}B_s &=& u_{1,1}^2u_{1,2}+\frac{u_{1,2}u_{2,1}^2}{\sigma}\\
  &=& u_{1,1}^2u_{1,2}+\frac{u_{2,1}}{\sigma}(u_{1,1}u_{2,2}-1)\\       
  &=& u_{1,1}^2u_{1,2}+\frac{u_{1,1}}{\sigma}(-u_{1,1}u_{1,2}\sigma)-\frac{u_{2,1}}{\sigma}\\
  &=& -\frac{u_{2,1}}{\sigma}.
\end{eqnarray*}

By Lemma \ref{lem:wij},
$u_{1,2}=w_{1,2}\tau_n$.
As remarked above Lemma \ref{lem:wij}, $w_{1,2}\ne 0$.
If $u_{1,2}=0$, then $\tau_n=0$.
But this implies
$\tau_{n+1}=0$, because $\phi_n(s,t)=\tau_{n+1}-(t+1/t-1-s)\tau_n=0$.
From the recursive relation, this implies $\tau_m=0$ for all $m$.
But this is absurd, because $\tau_1=1$.
Hence $u_{1,2}\ne 0$, so $B_s=-u_{2,1}/(u_{1,2}\sigma)$.
From Lemma \ref{lem:wij} again,
$u_{1,2}=w_{1,2}\tau_n$ and $u_{2,1}=w_{2,1}\tau_n$.
Thus $B_s=-w_{2,1}/(w_{1,2}\sigma)$.
\end{proof}


\section{Limits}

Let $r=p/q$ be a rational number, and let $M_n(r)$ denote
the resulting manifold by $r$-filling on the knot exterior $M_n$ of $K_n$.
In other words, $M_n(r)$ is obtained by 
attaching a solid torus $V$ to $M_n$ along their boundaries so that
the loop $x^p\mathcal{L}^q$ bounds a meridian disk of $V$, where
$x$ and $\mathcal{L}$ are a meridian and longitude of $K_n$.

Our representation $\rho_s: G\to SL_2(\mathbb{R})$ induces
a homomorphism $\pi_1(M_n(r))\to SL_2(\mathbb{R})$
if and only if $\rho_s(x)^p\rho_s(\mathcal{L})^q=I$.
Since both of $\rho_s(x)$ and $\rho_s(\mathcal{L})$ are diagonal
(see (\ref{eq:rho}) and Proposition \ref{prop:longi}),
this is equivalent to the equation
\begin{equation}\label{eq:slope}
A_s^p B_s^q=1,
\end{equation}
where $A_s$ and $B_s$ are the $(1,1)$-entries of $\rho_s(x)$ and $\rho_s(\mathcal{L})$, respectively.
We remark that $A_s=\sqrt{t}\ (>1)$ is a positive real number,
so is $B_s$ by Proposition \ref{prop:longi}.
Hence the equation (\ref{eq:slope}) is furthermore equivalent to the equation
\begin{equation}
-\frac{\log B_s}{\log A_s}=\frac{p}{q}.
\end{equation}

Let $g:(0,\infty)\to \mathbb{R}$ be a function defined by
\[
g(s)=-\frac{\log B_s}{\log A_s}.
\]

We will examine the image of $g$.

\begin{lemma}\label{lem:key-limit}
\begin{itemize}
\item[(1)] If $|n|>1$, then $\displaystyle\lim_{s\to+0}t=\infty$.
If $n=1$, then $\displaystyle\lim_{s\to+0}t=(3+\sqrt{5})/2$.
\item[(2)] $\displaystyle\lim_{s\to\infty}t=\infty$.
\item[(3)] $\displaystyle\lim_{s\to\infty}(t-s)=2$.
\item[(4)] $\displaystyle\lim_{s\to\infty}\frac{s}{t}=1$.
\end{itemize}
\end{lemma}

\begin{proof}
(1) 
If $n=1$, then the equation $\phi_1(s,T)=0$ has the unique solution
$T=(s^2+3s+3)/(s+1)$ for a given $s>0$, so $\lim_{s\to+0}T=3$.
Since $t=(T+\sqrt{T^2-4})/2$, we have $\lim_{s\to+0}t=(3+\sqrt{5})/2$.

Assume $|n|>1$.
From Proposition \ref{prop:root}, we have $s+2+c/s<T$, where $c$ is a positive constant.
Hence $\lim_{s\to+0}T=\lim_{s\to+0}t=\infty$.

(2) As $T>s+2$, $\lim_{s\to\infty}T=\lim_{s\to\infty}t=\infty$.

(3) Since $s+2<t+1/t<s+2+4/s$, (2) implies $\lim_{s\to\infty}(t-s)=2$.

(4) From $s+2<T<s+2+4/s$ again, we have $\lim_{s\to\infty}T/s=1$, which
implies $\lim_{s\to\infty}s/t=1$
\end{proof}

\begin{lemma}\label{lem:bslimit}
\begin{itemize}
\item[(1)] $\displaystyle\lim_{s\to+0}B_s=1$.
\item[(2)] $\displaystyle\lim_{s\to\infty}B_s\,t^2=1$.
\end{itemize}
\end{lemma}

\begin{proof}
(1) By Lemma \ref{lem:bs},
\[
B_s=-\frac{w_{2,1}}{w_{1,2}\sigma}=\frac{t-s-1}{-1+(1+s)t}.
\]
Lemma \ref{lem:key-limit}(1) implies $\lim_{s\to+0}B_s=1$.

(2) 
We decompose $B_s t^2$ as
\[
B_s t^2= (t-s-1)\cdot \frac{t^2}{-1+(1+s)t}.
\]
From Lemma \ref{lem:key-limit}(3) and (4),
\[
\lim_{s\to\infty}(t-s-1)=\lim_{s\to\infty}\frac{t^2}{-1+(1+s)t}=1.
\]
Hence $\lim_{s\to\infty}B_st^2=1$.
\end{proof}


\begin{proposition}\label{prop:g-image}
The image of $g$ contains an open interval $(0,4)$.
\end{proposition}

\begin{proof}
By Lemma \ref{lem:bslimit}(1),
$\lim_{s\to+0}\log B_s=0$.
Hence
\[
\lim_{s\to +0}g(s)=-\lim_{s\to +0}\frac{\log B_s}{\log A_s}=-\lim_{s\to+0}\frac{\log B_s}{\log\sqrt{t}}=0.
\]

Also, we have $\lim_{s\to\infty}(\log B_s+2\log t)=0$ by Lemma \ref{lem:bslimit}(2).
Thus
\[
\lim_{s\to \infty}g(s)=-\lim_{s\to\infty}\frac{\log B_s}{\log A_s}=
-\lim_{s\to\infty}\frac{2\log B_s}{\log t}=4.
\]
Hence the image of $g$ contains an interval $(0,4)$.
\end{proof}

A computer experiment suggests that the image of $g$ equals to $(0,4)$,
but we do not need this.


\section{Universal covering group}\label{sec:univ}

We briefly review the description of the universal covering group
of $SL_2(\mathbb{R})$.

Let
\[
SU(1,1)=\left\{
\begin{pmatrix}
\alpha & \beta \\
\bar{\beta} & \bar{\alpha}\\
\end{pmatrix}
\mid |\alpha|^2-|\beta|^2=1\right\}
\]
be the special unitary group over $\mathbb{C}$ of signature $(1,1)$.
It is well known that $SU(1,1)$ is conjugate to $SL_2(\mathbb{R})$ in $GL_2(\mathbb{C})$.
The correspondence is given by
$\psi:SL_2(\mathbb{R})\to SU(1,1)$, sending
$A\mapsto JAJ^{-1}$, where
\[J=
\begin{pmatrix}
1 & -i \\
1 & i
\end{pmatrix}.
\]
Thus
\[
\psi:
\begin{pmatrix}
a & b \\
c & d
\end{pmatrix}
\mapsto 
\begin{pmatrix}
\frac{a+d+(b-c)i}{2} & \frac{a-d-(b+c)i}{2} \\
\frac{a-d+(b+c)i}{2} & \frac{a+d-(b-c)i}{2}
\end{pmatrix}.
\]

There is a parametrization of $SU(1,1)$ by $(\gamma,\omega)$
 where $\gamma=\beta/\alpha$ and $\omega=\arg \alpha$ defined mod $2\pi$ (see \cite{B}).
Thus 
$SU(1,1)=\{(\gamma,\omega) \mid |\gamma|<1, -\pi\le \omega<\pi\}$.
Topologically, $SU(1,1)$ is an open solid torus $\Delta\times S^1$, where
$\Delta=\{\gamma\in \mathbb{C}\mid |\gamma|<1\}$. 
The group operation is given by
$(\gamma,\omega)(\gamma',\omega')=(\gamma'',\omega'')$, where
\begin{eqnarray}\label{eq:sl2R}
\gamma''&=& \frac{\gamma'+\gamma e^{-2i\omega'}}{1+\gamma\bar{\gamma'}e^{-2i\omega'}},\label{eq:sl2R1}\\
\omega''&=& \omega+\omega'+\dfrac{1}{2i}\log
\frac{1+\gamma\bar{\gamma'}e^{-2i\omega'}}{1+\bar{\gamma}\gamma'e^{2i\omega'}}.
\label{eq:sl2R2}
\end{eqnarray}
These equations come from the matrix operation.
Here, the logarithm function is defined by its principal value and $\omega''$ is defined by mod $2\pi$.
The identity element is $(0,0)$, and
the correspondence between
$\begin{pmatrix}
\alpha & \beta \\
\bar{\beta} & \bar{\alpha}\\
\end{pmatrix}$ and $(\gamma,\omega)$ gives an isomorphism.

Now, the universal covering group $\widetilde{SL_2(\mathbb{R})}$ of $SU(1,1)$
can be described as
\[
\widetilde{SL_2(\mathbb{R})}=\{(\gamma,\omega)\mid |\gamma|<1, -\infty<\omega<\infty\}.
\]
Thus $\widetilde{SL_2(\mathbb{R})}$ is homeomorphic to $\Delta\times \mathbb{R}$. 
The group operation is given by (\ref{eq:sl2R1}) and (\ref{eq:sl2R2}) again, but
$\omega''$ is not mod $2\pi$ anymore.

Let $\chi: \widetilde{SL_2(\mathbb{R})}\to SL_2(\mathbb{R})$ be the covering projection.
Then $\ker\chi=\{(0,2m\pi)\mid m\in \mathbb{Z}\}$.

\begin{lemma}
The subset $(-1,1)\times \{0\}$ of $\widetilde{SL_2(\mathbb{R})}$ forms a subgroup.
\end{lemma}

\begin{proof}
From (\ref{eq:sl2R1}) and (\ref{eq:sl2R2}),
it is straightforward to see that $(-1,1)\times \{0\}$ is closed under the group operation.
For $(\gamma,0)\in (-1,1)\times\{0\}$, its inverse is $(-\gamma,0)$.
\end{proof}

For the representation $\rho_s : G\to SL_2(\mathbb{R})$ defined by (\ref{eq:rho}),
\begin{equation}
\psi(\rho_s(x))=\frac{1}{2\sqrt{t}}
\begin{pmatrix}
t+1 & t-1 \\
t-1 & t+1
\end{pmatrix}\in SU(1,1).
\end{equation}
Thus $\psi(\rho_s(x))$ corresponds to $(\gamma_x,0)$, where $\gamma_x=(t-1)/(t+1)$.
Since $t>1$, $\gamma_x\in (-1,1)$.


Also, for the longitude $\mathcal{L}$,
\[
\psi(\rho_s(\mathcal{L}))=\frac{1}{2}
\begin{pmatrix}
B_s+\frac{1}{B_s} & B_s-\frac{1}{B_s} \\
B_s-\frac{1}{B_s} & B_s+\frac{1}{B_s}
\end{pmatrix}, B_s>0
\]
from Proposition \ref{prop:longi}.
Thus $\psi(\rho_s(\mathcal{L}))$ corresponds to $(\gamma_\mathcal{L},0)$, where
$\gamma_\mathcal{L}=(B_s^2-1)/(B_s^2+1)$.
Clearly, $\gamma_\mathcal{L}\in (-1,1)$.

\section{Proof of Theorem \ref{thm:main}}

Since the knot exterior $M_n$ of $K_n$ satisfies $H^2(M_n;\mathbb{Z})=0$,
any $\rho_s: G\to SL_2(\mathbb{R})$ lifts to a representation
$\tilde{\rho}: G\to\widetilde{SL_2(\mathbb{R})}$ (\cite{G}).
Moreover, any two lifts $\tilde{\rho}$ and $\tilde{\rho}'$ are
related as follows:
\[
\tilde{\rho}'(g)=h(g)\tilde{\rho}(g),
\]
where $h:G\to \ker \chi\subset\widetilde{SL_2(\mathbb{R})}$.
Since $\ker \chi=\{(0,2m\pi)\mid m\in \mathbb{Z}\}$ is isomorphic to $\mathbb{Z}$,
the homomorphism $h$ factors through $H_1(M_n)$, so
it is determined only by the value $h(x)$ of a meridian $x$ (see \cite{Kh}).

The following result is the key in \cite{BGW}, which is originally claimed in \cite{Kh}, for the figure eight knot. 
Our proof most follows that of \cite{BGW}, but it is much simpler,
because the values of $\psi(\rho_s(x))$ and $\psi(\rho_s(\mathcal{L}))$
are calculated explicitly in Section \ref{sec:univ}.

\begin{lemma}\label{lem:key}
Let $\tilde{\rho}: G\to \widetilde{SL_2(\mathbb{R})}$ be a lift of $\rho_s$.
Then replacing $\tilde{\rho}$ by a representation
$\tilde{\rho}'=h\cdot \tilde{\rho}$ for some $h:G\to \widetilde{SL_2(\mathbb{R})}$,
we can suppose that $\tilde{\rho}(\pi_1(\partial M_n))$ is contained in the subgroup $(-1,1)\times \{0\}$ of $\widetilde{SL_2(\mathbb{R})}$.
\end{lemma}

\begin{proof}
Since $\chi(\tilde{\rho}(\mathcal{L}))=(\gamma_\mathcal{L},0)$,
$\tilde{\rho}(\mathcal{L})=(\gamma_\mathcal{L},2j\pi)$ for some $j$.
On the other hand, $\mathcal{L}$ is a commutator,
because our knot is genus one.
Therefore the inequality (5.5) of \cite{W} implies  $-3\pi/2<2j\pi<3\pi/2$.
Thus we have $\tilde{\rho}(\mathcal{L})=(\gamma_\mathcal{L},0)$.

Similarly, $\tilde{\rho}(x)=(\gamma_x,2l \pi)$ for some $l$.
Let us choose $h:G\to \widetilde{SL_2(\mathbb{R})}$
so that $h(x)=(0,-2l\pi)$.
Set $\tilde{\rho}'=h\cdot \tilde{\rho}$.
Then a direct calculation shows that $\tilde{\rho}'(x)=(\gamma_x,0)$ and $\tilde{\rho}'(\mathcal{L})=(\gamma_\mathcal{L},0)$.
Since $x$ and $\mathcal{L}$ generate the peripheral subgroup $\pi_1(\partial M_n)$,
the conclusion follows from these.
\end{proof}

\begin{proof}[Proof of Theorem \ref{thm:main}]
For $r=0$, $M_n(0)$ is irreducible and has positive betti number.
Hence $\pi_1(M_n(0))$ is left-orderable by \cite[Corollary 3.4]{BRW}.
For $r=4$, \cite{CLW} and \cite{T} confirmed the conclusion.

Let $r=p/q\in (0,4)$.
By Proposition \ref{prop:g-image}, we can fix $s$ so that $g(s)=r$.
Choose a lift $\tilde{\rho}$ of $\rho_s$ so that
$\tilde{\rho}(\pi_1(\partial M_n))\subset (-1,1)\times\{0\}$.
Then $\rho_s(x^p\mathcal{L}^q)=I$, so $\chi(\tilde{\rho}(x^p\mathcal{L}^q))=I$.
This means that $\tilde{\rho}(x^p\mathcal{L}^q)$ lies in $\ker\chi=\{(0,2m\pi)\mid m\in \mathbb{Z}\}$.
Hence $\tilde{\rho}(x^p\mathcal{L}^q)=(0,0)$.
Then $\tilde{\rho}$ can induce a homomorphism $\pi_1(M_n(r))\to \widetilde{SL_2(\mathbb{R})}$
with non-abelian image.
Recall that $\widetilde{SL_2(\mathbb{R})}$ is left-orderable (\cite{Be}).
Hence any (non-trivial) subgroup of $\widetilde{SL_2(\mathbb{R})}$ is left-orderable.
Since $M_n(r)$ is irreducible (\cite{HT}), 
$\pi_1(M_n(r))$ is left-orderable by \cite[Theorem 1.1]{BRW}.
This completes the proof.
\end{proof}

\bibliographystyle{amsplain}

\end{document}